\documentclass{amsart}
\usepackage{amssymb,amsmath}

\newtheorem{theorem}{Theorem}[section]
\newtheorem{lemma}[theorem]{Lemma}

\theoremstyle{definition}
\newtheorem{definition}[theorem]{Definition}

\theoremstyle{remark}

\numberwithin{equation}{section}

\newcommand{\eps}{\varepsilon}

\def\rep{{\rm rep}}
\def\Rep{{\rm Rep}}
\def\rme{{\rm e}}

\begin{document}

\title[On the $b$-ary expansions of $\log (1 + \frac{1}{a})$ and $\rme$]{On 
the $b$-ary expansions of $\log (1 + \frac{1}{a})$ and $\rme$}

\author{Yann Bugeaud}
\address{Department of Mathematics, Universit\'e de Strasbourg, 7 rue Ren\'e
Descartes, 67084 Strasbourg, France}
\email{bugeaud@math.unistra.fr}

\author{Dong Han Kim}
\address{Department of Mathematics Education,
Dongguk University -- Seoul, Seoul 04620, Korea.}
\email{kim2010@dongguk.edu}

\begin{abstract}
Let $b \ge 2$ be an integer and $\xi$ an irrational real number.
We prove that, if the irrationality exponent of $\xi$ is equal to $2$ or slightly greater than $2$,
then the $b$-ary expansion of $\xi$ cannot be `too simple', in a suitable sense. 
Our result applies, among other classical numbers, to 
badly approximable numbers, non-zero rational powers of $\rme$, and
$\log (1 + \frac{1}{a})$, provided that the integer $a$ is sufficiently large. It 
establishes an unexpected connection between the irrationality exponent of a 
real number and its $b$-ary expansion. 
%The present paper illustrates the fruitful interplay between combinatorics on words and 
%Diophantine approximation, which has already led to several recent progress.  
\end{abstract}



\maketitle 

\section{Introduction and main result}

Throughout the present paper, $b$ always denotes an integer 
greater than or equal to $2$
and $\xi$ a real number. There exists a unique infinite sequence 
$(a_j)_{j \ge 1}$ of integers from $\{0, 1, \ldots, b-1\}$, called the
$b$-ary expansion of $\xi$, such that
$$
\xi = \lfloor \xi \rfloor + \sum_{j\ge 1} \, \frac{a_j}{b^j}    \eqno (1.1)
$$
and $a_j \not= b-1$ for infinitely many indices $j$. 
Here, $\lfloor \cdot \rfloor$ denotes the integer part function. 
Clearly, the sequence $(a_j)_{j \ge 1}$ is ultimately periodic if, and only if, $\xi$ is rational.

The real number $\xi$ is called {\it normal to base $b$} if, for any positive integer $k$,
each one of the $b^k$ blocks of $k$ digits from $\{0, 1, \ldots, b-1\}$ 
occurs in the $b$-ary expansion $a_1 a_2 \ldots $ of $\xi$ with the same frequency $1/b^k$.
The first explicit example of a number normal to base $10$, namely the number
$$
0.1234567891011121314\ldots,   
$$
whose sequence of digits is the concatenation of all
positive integers ranged in increasing order, 
was given in 1933 by Champernowne \cite{Ch33}; see the monograph \cite{BuLiv2}
for further results.  Almost all real
numbers (here and below, `almost all' always refers to the Lebesgue measure) 
are normal to every
base $b$, but proving that a specific number, like $\rme$, $\pi$, $\sqrt{2}$
or $\log 2$ is normal to some base remains a challenging open 
problem, which seems to be completely out of reach.

In the present paper, we focus our
attention to apparently simpler questions. We take a point of view
from combinatorics on words. 
Let $\mathcal A$ be a finite set called an alphabet and denote by $|\mathcal A|$ its cardinality.
A word over $\mathcal A$ is a finite or infinite sequence of elements of $\mathcal A$.
For a (finite or infinite) word ${\bf x} = x_1 x_2 \ldots$ written over $\mathcal A$,
let $n \mapsto p (n, {\bf x})$ denote its subword complexity function
which counts the number of different subwords of length $n$ occurring in $\mathbf x$, that is,
$$
p (n,{\bf x}) = \# \{ x_{j+1} x_{j+2} \dots x_{j+n} : j \ge 0 \}, \quad n \ge 1.
$$
Clearly, we have
$$
1 \le p(n, {\bf x}) \le |\mathcal A|^n, \quad n \ge 1.
$$
If ${\bf x}$ is ultimately periodic, then there exists an integer $C$ such that 
$p(n, {\bf x}) \le C$ for $n \ge 1$. Otherwise, we have
$$
p(n+1, {\bf x}) \ge p(n, {\bf x}) + 1, \quad n \ge 1, \eqno (1.2)
$$
thus $p(n, {\bf x}) \ge n+1$ for $n \ge 1$. There exist uncountably many infinite words 
${\bf s}$ over $\{0, 1\}$ such that $p(n, {\bf s}) = n+1$ for $n \ge 1$. These words are called
Sturmian words. Classical references on combinatorics on words and
on Sturmian sequences include \cite{Fogg02,Loth02,AlSh03}.

A natural way to measure the
complexity of the real number $\xi$ written in base $b$ 
as in (1.1) is to count the number of
distinct blocks of given length in the infinite word ${\bf a} = a_1 a_2 \ldots$. 
%Obviously, ${\bf a}$ depends on $\xi$ and $b$, 
%but we choose not to indicate this dependence: this notation will be kept throughout the paper.
Thus, for $n \ge 1$, we set $p(n, \xi, b) = p(n, {\bf a})$. 
%for an infinite word ${\bf w}$ over
%the alphabet $\{0, 1, \ldots, b-1\}$ and for any positive integer $n$,
%we let $p(n, {\bf w})$ denote the number of distinct blocks of $n$ letters
%occurring in ${\bf w}$. Furthermore, 
%with ${\bf a}$ as above. 
Obviously, we have
$$
p(n, \xi, b) = \# \{a_{j+1} a_{j+2} \ldots a_{j+n} : j \ge 0\}, \quad n \ge 1, 
$$
and
$$
1 \le p(n, \xi, b) \le b^n, \quad n \ge 1, 
$$
where both inequalities are sharp. 

If $\xi$ is normal to base $b$, then $p(n, \xi, b) = b^n$ for every 
positive integer $n$. Clearly, the converse does not always hold. 
To establish a good lower bound for $p(n, \xi, b)$ is a first step towards 
the confirmation that the real number $\xi$ is normal to base $b$.
This point of view has been taken by Ferenczi and Mauduit \cite{FeMa97} in 1997. 
It follows from their approach (see also \cite{All00}) that we have
$$
\lim_{n \to + \infty} \, 
\bigl( p(n, \xi, b) - n \bigr) = + \infty,  
$$
for every algebraic irrational number $\xi$ and every integer $b \ge 2$. 
Subsequently, by means of a new transcendence
criterion established in \cite{ABL}, their result 
was improved in \cite{AdBu07d} as follows. 

\begin{theorem}\label{complalg}
For every integer $b \ge 2$, every algebraic irrational number $\xi$ satisfies   
$$
\lim_{n\to+\infty}\, \frac{p(n, \xi, b)}{n} = +\infty.
$$
\end{theorem}

Much less is known for specific transcendental numbers. 
The only result available so far was obtained in \cite{Ad10} 
(see also Section 8.5 of \cite{BuLiv2}), 
as the consequence of two combinatorial statements established in
\cite{BeHoZa06} and \cite{AdBu11} on the structure of 
Sturmian words.
%, which are the non ultimately periodic words with the smallest subword complexity.
%(see Definitions xxx and xxx below).
Before stating this result, we recall a basic notion from
Diophantine approximation.

\begin{definition}\label{defirrexp}
The irrationality exponent $\mu(\xi)$ of an irrational real number $\xi$ is the supremum
of the real numbers $\mu$ such that the inequality
$$
\biggl| \xi - \frac{p}{q} \biggr| < \frac{1}{q^{\mu}}
$$
has infinitely many solutions in rational numbers $\frac{p}{q}$.
%If $\mu(\xi)$ is infinite, then $\xi$ is called a Liouville number. 
\end{definition}

The theory of continued fraction implies that 
every irrational real number $\xi$ satisfies $\mu(\xi) \ge 2$. Combined with an
easy covering argument, we get that the irrationality exponent of almost 
every real number is equal to $2$.
Theorem 1 of \cite{Ad10}, reproduced below, extends the result of Ferenczi and 
Mauduit mentioned above to 
real numbers whose irrationality exponent is equal to $2$
(recall that, by Roth's theorem \cite{Roth55}, the irrationality exponent of 
every real algebraic irrational number is equal to $2$).

\begin{theorem}\label{complmu2}
For every integer $b \ge 2$,
every irrational real number $\xi$ whose irrationality exponent is equal to $2$ satisfies
$$
\lim_{n \to + \infty} \, \bigl( p(n, \xi, b) - n \bigr) = + \infty.
$$
\end{theorem}

Theorem \ref{complmu2}  applies to a wide class of classical numbers, 
including non-zero rational powers of $\rme$, badly approximable numbers, 
$\tan \frac{1}{a}$, where $a$ is a positive integer, etc.
Further examples of real numbers whose irrationality exponent is known
to be equal to $2$ are listed in \cite{Ad10}. 

Theorem \ref{complmu2} covers all what is
known at present on the $b$-ary expansion of transcendental numbers. 
The main result of the present paper is the following considerable
improvement of Theorem \ref{complmu2}.

\begin{theorem}\label{complmain}
Let $b \ge 2$ be an integer and $\xi$ an irrational real number.
If $\mu$ denotes the irrationality exponent of $\xi$, then
$$
\liminf_{n \to + \infty} \, \frac{p(n, \xi, b)}{n} \ge 
1 +  \frac{1 - 2 \mu (\mu - 1) (\mu - 2)}{\mu^3 (\mu - 1)}.    \eqno (1.3)
%\frac{\mu^4 - 3 \mu^3 + 6 \mu^2 - 4 \mu +1 }{\mu^3 (\mu - 1)}
$$
and
$$
\limsup_{n \to + \infty} \, \frac{p(n, \xi, b)}{n}  
\ge  1  +  \frac{1 - 2 \mu (\mu - 1) (\mu - 2)}{3 \mu^3 - 6 \mu^2 + 4 \mu - 1}. \eqno (1.4)
$$
In particular, every irrational real number $\xi$ whose irrationality exponent
is equal to $2$ satisfies
$$
\liminf_{n \to + \infty} \, \frac{p(n, \xi, b)}{n}  \ge \frac{9}{8} 
\quad \hbox{and} \quad  
\limsup_{n \to + \infty} \, \frac{p(n, \xi, b)}{n}  \ge \frac{8}{7}.    \eqno (1.5) 
$$
\end{theorem}

We display an immediate consequence of Theorem \ref{complmain}.

\begin{theorem}
For any integer $b \ge 2$ we have
$$
\liminf_{n \to + \infty} \, \frac{p(n, \rme, b)}{n}  \ge \frac{9}{8} 
\quad \hbox{and} \quad  
\limsup_{n \to + \infty} \, \frac{p(n, \rme, b)}{n}  \ge \frac{8}{7}.    
$$
\end{theorem}

Theorem \ref{complmain} establishes an unexpected 
connection between the irrationality exponent of a 
real number and its $b$-ary expansion. 
It gives a non-trivial result on the $b$-ary expansion of
a real number $\xi$ when $2 \le \mu (\xi) <  2.1914  \ldots$  
It applies to a much wider class
of classical numbers than Theorem \ref{complmu2}, which includes in particular 
the transcendental number $\log (1 + \frac{1}{a})$, where 
$a$ is a large positive integer. More examples are given in Section 2.  
Theorem \ref{complmain} is sharp up to the values of the numerical constants 
occurring in (1.3) to (1.5).

The present paper illustrates the fruitful interplay between combinatorics on words and 
Diophantine approximation, which has already led recently to several recent progress. 
The proof of Theorem \ref{complmain}, given in Section 3, is mostly combinatorial and 
essentially self-contained. 

%%%%%%%%%%%%%%%%%%%%%%%%%%%%%%%%%%%%%%%%%%%%%

\section{A further result, comments, and examples}
%{On the $b$-ary expansions of $\log (1 + \frac{1}{a})$ and $\rme$}

%Our main result is not expressed in terms of the subword complexity function, but
%involves another 
A key ingredient for the proof of Theorem \ref{complmain} 
is the study of a complexity function 
which takes into account the smallest return time of a factor of an infinite word. 
For an infinite word ${\bf x} = x_1 x_2 \dots $ and an integer $n \ge 1$, set 
$$ 
r(n,{\bf x}) = \min \{ m \ge 1 :   x_{i} \ldots x_{i+n-1} = x_{m-n+1} \ldots x_{m} 
\text{ for some } i \text{ with } 1 \le i \le m-n \} .
$$
Said differently, $r(n,\mathbf x)$ denotes the length of the smallest prefix of ${\bf x}$
containing two (possibly overlapping) occurrences of some word of length $n$. 
The function $n \mapsto r(n, {\bf x})$ has been 
introduced and studied in \cite{BuKim15a}, where the following two assertions 
are established. For every infinite word ${\bf x}$ which is not ultimately periodic, there exist
arbitrarily large integers $n$ such that $r(n,{\bf x}) \ge 2n+1$. The only 
infinite words ${\bf x}$ such that $r(n,{\bf x}) \le 2n+1$ for $n \ge 1$
and which are not ultimately periodic are the Sturmian words.

Let $\xi$ be an irrational real number and $b \ge 2$ be an integer. 
Write $\xi$ in base $b$ as in (1.1) and set ${\bf a} = a_1 a_2  \ldots$ 
%the infinite word 
%corresponding to the $b$-ary expansion of $\xi$ given in (1.1). 
For $n \ge 1$, set $r(n, \xi, b) = r(n, {\bf a})$.
The following result asserts that, if the irrationality exponent of $\xi$
is not too large, then the function $n \mapsto r(n, \xi, b)$ cannot increase too slowly.

\begin{theorem}\label{complmainrep}
Let $b \ge 2$ be an integer and $\xi$ an irrational real number.
If $\mu$ denotes the irrationality exponent of $\xi$, then
$$
\limsup_{n \to + \infty} \, \frac{r(n, \xi, b)}{n}  
%\ge  1 + \frac{\mu^3}{3 \mu^3 - 6 \mu^2 + 4 \mu - 1}.
\ge  2  +  \frac{1 - 2 \mu (\mu - 1) (\mu - 2)}{3 \mu^3 - 6 \mu^2 + 4 \mu - 1}.  \eqno (2.1)
$$
In particular, every irrational real number $\xi$ whose irrationality exponent
is equal to $2$ satisfies
$$
\limsup_{n \to + \infty} \, \frac{r(n, \xi, b)}{n}  \ge \frac{15}{7}.   \eqno (2.2)
$$
\end{theorem}

By Lemma \ref{ubound} below, 
$p(n, \xi, b) \ge r(n, \xi , b) - n$ holds for all integers $n \ge 1$, $b \ge 2$ and
every irrational real number $\xi$. Thus, (1.4) and the second assertion of (1.5)
are immediate consequences of (2.1) and (2.2), respectively.

%Note also that (1.2) shows that Theorem \ref{complmain}
%implies Theorem \ref{complmu2}. 

\medskip

We will establish Theorems \ref{complmain} and \ref{complmainrep} 
simultaneously in Section 3. Our key ingredient is a purely combinatorial  
auxiliary result, stated as Theorem \ref{diffReprep} below. 

\medskip

We stress that, even for real numbers whose irrationality exponent is equal to $2$, 
Theorem \ref{complmain} improves Theorem \ref{complmu2}. Indeed, 
Aberkane \cite{Abe03} proved the existence of infinite words ${\bf x}$ with the property that 
$$
\lim_{n \to + \infty} \, p(n, {\bf x}) - n = + \infty \quad   
\hbox{and} \quad
\lim_{n \to + \infty} \, \frac{p(n, {\bf x})}{n} = 1.
$$
Furthermore, he established in \cite{Abe01} that, for any real number $\delta$
with $\delta > 1$, there are infinite words ${\bf x}$
satisfying
$$
1 < \liminf_{n \to + \infty} \, \frac{p(n, {\bf x})}{n}
< \limsup_{n \to + \infty} \, \frac{p(n, {\bf x})}{n} \le \delta.
$$
See also Heinis \cite{Hei01,Hei02} for further results on words with small 
subword complexity. 

\medskip

Independently, Kmo\v{s}ek \cite{Km79} and Shallit \cite{Sh79} 
(see also Section 7.6 of \cite{BuLiv2}) established that 
the real number $\xi_{{\rm KS}} := \sum_{k \ge 1} \, 2^{-2^k}$ has a bounded continued fraction
expansion. In particular, it satisfies $\mu (\xi_{{\rm KS}}) = 2$. Since
$$
\limsup_{n \to + \infty} \, \frac{r(n, \xi_{{\rm KS}}, 2)}{n} = \frac{5}{2} \quad
\hbox{and} \quad
\liminf_{n \to + \infty} \, \frac{p(n, \xi_{{\rm KS}}, 2)}{n} = \frac{3}{2}, 
$$
this shows that the value $\frac{15}{7}$ in (2.2) 
cannot be replaced by a real number greater than $\frac{5}{2}$. 
%Furthermore, as noted in \cite{Ad10}, 
%one can deduce from the proof of Lemma 2.4 in \cite{Gh07} 
%that $p(n, \xi, 2) \le (2 + \log 3) n + 4$, for $n \ge 1$. 
Also, the value $\frac{9}{8}$ in (1.5) 
cannot be replaced by a real number greater than~$\frac{3}{2}$. 
Actually, with some additional effort and a case-by-case analysis, it is 
possible to replace the value $\frac{15}{7}$ in (2.2) 
and $\frac{9}{8}$ in (1.5)  by slightly larger numbers; see
the additional comments at the end
of Section 3. We have, however, chosen to present an elegant, short proof 
of Theorem \ref{complmainrep}, rather than a more complicated proof of a slightly sharper
version of it. 

It has been proved in \cite{Bu08} (see also Section 7.6 of \cite{BuLiv2}) that, for every 
real number $\mu \ge 2$, the irrationality exponent of $\xi_\mu :=   
\sum_{k \ge 1} \, 2^{-\lfloor \mu^k \rfloor}$ is equal to $\mu$. Since 
$p(n, \xi_{\mu}, 2) = O(n)$, this shows that Theorems \ref{complmain} and \ref{complmainrep} 
are best possible up to the values of the numerical constants.   

\medskip

Any real number whose sequence of partial quotients is bounded 
has its irrationality exponent equal to $2$, thus it 
satisfies (1.5) and (2.2), 
%This means that, if an irrational real number has a `too simple' continued fraction
%expansion, then 
and its expansion in an integer base $b$ cannot be `too simple'.

\medskip

Theorems \ref{complmain} and \ref{complmainrep} give 
non-trivial results on the $b$-ary expansion of
a real number $\xi$ when $2 \le \mu (\xi) <  2.1914 \ldots$ 
By means of a specific analysis of repetitions in Sturmian words, we were able 
in \cite{BuKim15a} to extend Theorem \ref{complmu2} 
to real numbers whose irrationality exponent
is less than or equal to $\frac{5}{2}$.     
Note that if ${\mathbf f} = f_1 f_2 \ldots$ 
denotes the Fibonacci word ${\mathbf f} = 01001010 \ldots$ (that is, the
fixed point of the substitution $0 \mapsto 01, 1 \mapsto 0$; this is a
Sturmian word), then the real number 
$\xi_{\mathbf f} := \sum_{k \ge 1} \, 2^{-f_k}$ satisfies 
$\mu(\xi_{\mathbf f}) = \frac{3 + \sqrt{5}}{2} = 2.618 \ldots$ and 
$p(n, \xi_{\mathbf f}, 2) = n+1$ for $n \ge 1$. 

\medskip

An important feature of Theorems \ref{complmain} and \ref{complmainrep}
is that they apply not only to real numbers whose irrationality exponent is equal to $2$, but also
to real numbers whose irrationality exponent is slightly larger than $2$. 
To prove that the irrationality exponent of a given real number is equal to $2$ is often a
very difficult problem, while it is sometimes possible to bound its value from above.
%allows us to get some information on the
%$b$-ary expansion of several classes of real numbers, 
%without knowing their continued fraction expansion. 
%Recall that  
For example, Alladi and Robinson \cite{AlRo80} 
(who improved earlier results of A. Baker \cite{Bak64}) 
and Danilov \cite{Da78} 
established that, for any positive integer $s$, the irrationality exponents of
$\log (1 + \frac{s}{t})$ and $\sqrt{t^2 - s^2} \arcsin \frac{s}{t}$ 
are bounded from above by a function of $t$ which tends to $2$ as the integer $t$
tends to infinity.  The next statement then follows at once from
Theorem \ref{complmain}.

\begin{theorem}\label{thmexamplesbis} 
Let $\eps$ be a positive real number.
For any positive integer $s$, there exists an integer $t_0$ such that,
for any integer $t > t_0$, we have
$$
%\limsup_{n \to + \infty} \, 
%\frac{r \bigl(n, \log \bigl(1 + \frac{s}{t} \bigr), b \bigr)}{n} \ge \frac{15}{7} - \eps,
%\quad
\liminf_{n \to + \infty} \, 
\frac{p \bigl(n, \log \bigl(1 + \frac{s}{t} \bigr), b \bigr)}{n} \ge \frac{9}{8} - \eps
$$
and
$$
%\limsup_{n \to + \infty} \, \frac{r (n, \sqrt{t^2 - s^2} \arcsin \frac{s}{t}, b)}{n} \ge \frac{15}{7} - \eps,
%\quad 
\liminf_{n \to + \infty} \, \frac{p (n, \sqrt{t^2 - s^2} \arcsin \frac{s}{t}, b)}{n} \ge \frac{9}{8} - \eps. 
$$
\end{theorem}

Using the results from \cite{Da78,AlRo80}, 
it is easy to give a suitable explicit value for $t_0$ in terms of $s$ and $\eps$. 
In particular, there exists an absolute positive constant $c$ such that, if $s, t$ 
are integers with $s \ge 2$ and $t \ge s^c$, then  
$$
\liminf_{n \to + \infty} \, 
\frac{p \bigl(n, \log \bigl(1 + \frac{s}{t} \bigr), b \bigr)}{n} \ge \frac{9}{8} - 4 \, \frac{\log s}{\log t}. 
$$
Up to now, not a single result was known on the $b$-ary expansion of 
the transcendental real number $\log (1 + \frac{1}{a} )$.

%%%%%%%%%%%%%%%%%%%%%%%%%%%%%%%%%%%%%%%%%%%%%

\section{Proofs}

We start with establishing a relationship between 
the subword complexity function of an infinite word ${\bf x}$ and
the function $n \mapsto r(n, {\bf x})$. 

Here and below, for integers $i, j$ with $1 \le i \le j$, we write 
$x_i^j$ for the factor $x_i x_{i+1} \ldots x_j$ of the word ${\bf x} = x_1 x_2 \ldots$

\begin{lemma}\label{ubound}
For any infinite word ${\mathbf x}$ and any positive integer $n$, we have
$$
p(n,{\mathbf x}) \ge r(n,{\mathbf x}) - n. 
$$
\end{lemma}

\begin{proof}
It follows from the definition of $r(n,{\mathbf x})$ that
the $r(n,{\mathbf x})-1- (n-1)$ factors of length $n$
of $x_1^{r(n,{\mathbf x})-1}$ are all distinct. 
Since $x_{r(n,{\mathbf x})-n+1}^{r(n,{\mathbf x})}$ 
is a factor of $x_1^{r(n,{\mathbf x})-1}$, we have
\begin{equation*}
p(n,{\mathbf x}) \ge p(n,x_1^{r(n,{\mathbf x})-1}) 
= p(n,x_1^{r(n,{\mathbf x})}) =  r(n,{\mathbf x})-n. \qedhere
\end{equation*}
\end{proof}

We stress that there is no analogue lower bound 
for the subword complexity function of ${\mathbf x}$
in terms of $n \mapsto r(n,{\mathbf x})$. 

For our combinatorial analysis, it is convenient to introduce two 
combinatorial exponents which measure the repetitions in an infinite word.

\begin{definition}
Let $\mathbf x$ be an infinite word. 
The exponent of repetition of $\mathbf x$, denoted by $\rep({\bf x})$, is the quantity
$$
\rep ({\mathbf x}) = \liminf_{n \to + \infty} \, \frac{r(n,{\mathbf x})}{n}. 
$$
The uniform exponent of repetition of $\mathbf x$, denoted by $\Rep({\bf x})$, is the quantity
$$
\Rep ({\mathbf x}) = \limsup_{n \to + \infty} \, \frac{r(n,{\mathbf x})}{n}. 
$$
\end{definition}

The key ingredient for the proof of Theorem \ref{complmainrep} is the following
%statement, which gives a non-trivial 
combinatorial theorem. 
%lower bound for the difference $\Rep({\mathbf x}) - \rep({\mathbf x})$.

\begin{theorem}\label{diffReprep}
Any infinite word $\mathbf x$ which is not ultimately periodic satisfies
$\Rep({\mathbf x}) \ge 2$,
%and
$$
\Rep({\mathbf x}) \ge \rep({\mathbf x}) + \frac{1}{1 + \rep({\mathbf x}) + (\rep({\mathbf x}))^2}, 
\eqno (3.1)
$$
and
$$
\liminf_{n \to + \infty} \, \frac{p(n,{\mathbf x})}{n} \ge 
\rep({\mathbf x}) - 1 + \frac{1}{(\rep({\mathbf x}))^3}.   \eqno (3.2)
$$
\end{theorem}

\begin{proof}
The first assertion of Theorem \ref{diffReprep} has been established
in \cite{BuKim15a}. 
It only remains for us to prove (3.1) and (3.2). 

Let ${\bf x} = x_1 x_2 \ldots $ be an infinite word which is not ultimately  periodic.  
Without any loss of generality, we may assume that 
$\rep({\mathbf x})$ is finite. Set $\rho = \rep({\mathbf x})$. 
Since (3.1) and (3.2) trivially hold for $\rho \le \frac{8}{5}$, we also assume that $\rho > \frac{8}{5}$. 

Let $\eps$ be a positive real number with $\eps < \frac{1}{10}$
and $n_0 \ge 3 \frac{\rho^2}{\eps}$  %%d I came back
be such that 
$$
(\rho - \eps) n \le r(n,{\mathbf x}), \quad \hbox{for $n \ge \frac{n_0}{8 \rho}$}. %%d
$$ 
%Without any loss of generality, we assume that $\rho - \eps \le 3$ and $\rho \ge \frac{3}{2}$. 

By Theorem 2.3 of \cite{BuKim15a}, there are arbitrarily large integers $n$
such that $r(n+1, {\bf x}) \ge r(n, {\bf x}) + 2$.  
Let $n > n_0$ be an integer such that $r(n+1,{\mathbf x}) > r(n,{\mathbf x}) + 1$
and define $\alpha$ by setting $r(n,{\mathbf x}) = \alpha n$. 
This implies that the word $x_{(\alpha - 1) n +1}^{\alpha n}$ of length $n$ has two
occurrences in $x_1^{\alpha n}$ and that
these two occurrences are not followed by the same letter.  
Let $m_1$ be the index at which the first occurrence of $x_{(\alpha - 1) n +1}^{ \alpha  n}$
starts. We have $m_1 + n - 1 <  \alpha n$ and the letters 
$x_{m_1 + n}$ and $x_{\alpha n + 1}$ are different.

Let $\beta$ be such that $r(n+1,{\mathbf x}) = \beta (n+1)$. 
Since $r(n+1, {\mathbf x}) \ge r(n, {\mathbf x}) + 2$, we have 
$\beta (n+1) \ge   \alpha n + 2$, that is $1 + (\beta - 1) (n+1) \ge (\alpha - 1) n + 2$. 
Then, the word $x_{(\beta - 1)(n+1)+1}^{\beta (n+1)}$ of length $n+1$
has two occurrences in $x_1^{\beta (n+1)}$. 
Let $m_2$ be the index at which its first occurrence starts. 
We have $m_2 < (\beta - 1)(n+1)+1$. 

If $\alpha \ge \rho + 2$, then $\beta (n+1) \ge (\rho+2) n + 2$ and we deduce that $\beta \ge \rho +1$ since $n  > n_0 > \rho +1$.  %%d

We assume that $\alpha < \rho +2$ and 
$$
\frac{1 - \beta + \alpha - \eps}{\beta - 1} > \frac{1+\rho}{(\rho - \eps)^2}  \eqno (3.3)
$$
and we will get a contradiction.

Consider the word $V_n := x_{(\beta - 1)(n+1)+1}^{\alpha n}$ 
of length 
$$
v_n = (1 - \beta + \alpha) n  - \beta +1. 
$$

Observe that $\rho - \eps > \frac{8}{5} - \frac{1}{10} \ge \frac{3}{2}$ implies that 
$\beta \ge \frac{3}{2}$ and check that, by (3.3), 
$$
v_n \ge (\beta - 1) \frac{1 + \rho}{\rho^2} n - (\beta - 1) \ge 
\frac{1}{2} \Bigl( \frac{n}{\rho} - 1 \Bigr) \ge \frac{n}{4 \rho},
$$
since $n \ge 2 \rho$. 

The word $V_n$ is a proper suffix of $x_{(\alpha - 1) n +1}^{ \alpha  n}$ and we have 
$$
V_n = x_{(\beta - 1)(n+1)+1}^{\alpha  n} 
= x_{m_2}^{m_2+v_n-1} = x_{m_1 + n - v_n}^{m_1 + n - 1}.
$$
If $m_2 = m_1 + n - v_n$, then $x_{m_2 + v_n} = x_{m_1 + n}$ and we deduce from 
$x_{m_2 + v_n} = x_{\alpha  n + 1}$ that $x_{m_1 + n} = x_{ \alpha  n + 1}$, a 
contradiction with our choice of $n$. 
Consequently, the word $V_n$ 
has (at least) three occurrences in $x_1^{\alpha  n}$. Set
$$
j_3 = (\beta - 1) (n+1)+1.
$$
Let $j_1, j_2$ with $j_1 < j_2 < j_3$ be the indices at which the two other
occurrences of $x_{j_3}^{\alpha  n}$ start. In particular, the letters 
$x_{j_1 + v_n}$ and $x_{j_2 + v_n}$ must be different.

%{\bf There are two cases to distinguish, see Section 4}

The proof decomposes into five steps. We show that $j_2$ and $j_1$ cannot be 
too small and that the three occurrences of $V_n$ in $x_1^{\alpha n}$ 
overlap. We conclude in Step 5 that 
the letters $x_{j_1 + v_n}$ and $x_{j_2 + v_n}$ must be the same. This contradiction
shows that (3.3) cannot hold. 

For a finite word $W$ and a real number $t > 1$, we denote by $(W)^t$ 
the word equal to the concatenation of $\lfloor t \rfloor$ copies of the word $W$ followed by the
prefix of $W$ of length $\lceil t - \lfloor t \rfloor \rceil$ times the length of $W$, where 
$\lceil x \rceil$ denotes the smallest integer greater than or equal to $x$. We say
that $(W)^t$ is the $t$-th power of $W$.  

\medskip

Step 1. 
Since $v_n \ge \frac{n}{4 \rho}$, our choice of $n_0$ implies that 
$$
(\rho - \eps) v_n \le r(v_n,{\mathbf x}) \le j_2 +  v_n -1,
$$ 
thus we get
$$
j_2 \ge (\rho - 1 - \eps) v_n +1.     \eqno (3.4)
$$
We have established that $j_2$ cannot be too small. 

\medskip

Step 2. 
Since $j_2$ is not too small, the subwords 
$x_{j_3}^{\alpha n} = x_{j_3}^{j_3+v_n - 1}$ and $x_{j_2}^{j_2+v_n - 1}$
(which are both equal to $V_n$) 
have a quite big overlap. Consequently, by Theorem~1.5.2 of \cite{AlSh03}, 
the word $V_n$ is the $t$-th power with
$$
t := \frac{v_n}{j_3-j_2} 
%- 1 + \frac{(2- \alpha)n -j_2 +1}{j_3 -j_2} 
%= \frac{(2- \alpha)n - j_3+1}{j_3 - j_2} = \frac{v_n}{j_3-j_2}
$$
of some word $U_n$ of length $u_n := j_3 - j_2$. 
We have $x_{j_2}^{j_3+v_n - 1} = (U_n)^{1+t}$. 
Observe that $n+1 > n_0 \ge \frac{3 \rho^2}{\eps} >\frac{\rho +2}{\eps}> \frac{\alpha}{\eps}$,
thus $v_n \ge (1 - \beta + \alpha - \eps) (n+1)$ and, by (3.4), 
\begin{align*}
t & \ge \frac{v_n}{(\beta - 1)(n+1) - (\rho - 1- \eps) v_n} \\ 
& \ge
\frac{1 - \beta + \alpha - \eps}{\beta - 1 - (1 - \beta + \alpha - \eps) (\rho - 1 - \eps)} \\
& \ge \frac{1 + \rho}{(\rho- \eps)^2 - (\rho - 1 - \eps) (1 + \rho)} 
\ge \frac{1 + \rho}{1 + \eps + \eps^2}.
\end{align*}
%if $\eps$ is small enough, by (3.3).
Recalling that $\rho \ge \frac{8}{5}$ and $\eps \le \frac{1}{10}$, we have established that 
$t \ge \frac{9}{4}$.

\medskip

Step 3. 
Let $W_n$ be the word such that $V_n = U_n W_n$ and let $w_n$
denote its length. Observe that
$$
w_n = \frac{t-1}{t} v_n =  v_n - j_3 + j_2    \eqno (3.5)
$$
and, recalling that $v_n \ge \frac{n}{4 \rho}$ and $t \ge \frac{9}{4}$, 
$$
w_n = \frac{t-1}{t} v_n \ge \frac{5}{9} \cdot \frac{n}{4 \rho} \ge \frac{n}{8 \rho}. 
$$
Since $V_n = (U_n)^t$ and $t > 2$, the word 
$W_n$ is a prefix of $V_n$ and it 
has two occurrences in the prefix of ${\bf x}$ of length $j_1 + v_n - 1$. 
It then follows from our choice of $n_0$ that 
$$
(\rho - \eps) w_n \le r( w_n, {\mathbf x}) \le j_1 + v_n -1.
$$
Combined with (3.5), this gives 
$$
j_1 \ge (\rho - 1 - \eps) v_n - (\rho - \eps) (j_3 - j_2) +1.  \eqno (3.6)
$$
We have established that $j_1$ cannot be too small. 

\medskip

Step 4. 
Observe first that (3.3) is equivalent to the inequality
$$
(\rho - \eps)^2 (1 - \beta + \alpha - \eps) > (\beta - 1) (\rho + 1).
$$
This gives
\begin{align*}
(\rho - \eps)^2 (1 - \beta + \alpha) n - (\rho + 1 - \eps) (\beta - 1) n
&  >  n \eps (\beta - 1) \\
& > (\beta - 1) [ (\rho - \eps)^2 + \rho + 1 - \eps],
\end{align*}
since $n \eps > n_0 \eps \ge 3 \rho^2$. %%d
Consequently, we get
$$
(\rho - \eps)^2 v_n > (\rho + 1 - \eps) (\beta - 1) (n + 1) = (\rho + 1 - \eps) (j_3 - 1).  \eqno (3.7) 
$$
We deduce from (3.4) that
$$
(\rho - \eps)^2 v_n \le (\rho - \eps) v_n + (\rho - \eps) (j_2 - 1),
$$
which, combined with (3.7), gives
\begin{align*}
(\rho - \eps) v_n & \ge  (\rho + 1 - \eps) (j_3 - 1) -  (\rho - \eps) (j_2 - 1) \\
& =  (\rho - \eps) (j_3 - j_2)  + j_3 - 1.
\end{align*}
%$$
%v_n > j_3 - (\rho - 1 - \eps) v_n + (\rho - \eps) (j_3 - j_2) - 1.
%$$
We conclude by (3.6) that
$$
v_n > j_3 - j_1.  \eqno (3.8)
$$
Thus, the subwords $x_{j_1}^{j_1+v_n-1}$ and $x_{j_3}^{j_3+v_n-1}$,
which are both equal to $V_n$, overlap.

\medskip

Step 5. 
It follows from (3.8) that
$$
v_n - (j_2 - j_1) > j_3 - j_2 = u_n,
$$
which means that the length of the overlap 
between the subwords $x_{j_1}^{j_1+v_n-1}$ and $x_{j_2}^{j_2+v_n-1}$
exceeds the length $u_n$ of $U_n$. We show that 
this implies that $x_{j_1}^{(2 - \alpha)n} = x_{j_1}^{j_3 + v_n - 1}$ 
is equal to a (large) power of some word.
To do this, we distinguish two cases.

If there exists an integer $h$ such that $j_2 = j_1 + h u_n$, then we have
$$
x_{j_1}^{j_3+v_n-1} = x_{j_1}^{j_2-1} \, x_{j_2}^{j_3+v_n-1} = (U_n)^{h+1+t}
$$
and the letters $x_{j_1+v_n}$ and $x_{j_2+v_n}$
are the same, since $j_1+v_n$ and $j_2+v_n$ are congruent modulo 
the length $u_n$ of $U_n$. This is a contradiction.

If $j_2 - j_1$ is not an integer multiple of $u_n$, then let $h$ be the smallest integer
such that $j_1 + h u_n > j_2$. The word $Z_n := x_{j_2}^{j_1 + h u_n - 1}$ 
is a suffix of $U_n$ and the word 
$Z'_n := x_{j_1 + h u_n}^{j_2 + u_n - 1} = x_{j_1 + h u_n}^{j_3 - 1}$ is a 
prefix of $U_n$. They satisfy
$$
U_n = Z_n Z'_n = Z'_n Z_n.
$$
By Theorem 1.5.3 of \cite{AlSh03}, the words $Z_n$ and $Z'_n$ are integer powers of a same
word. Thus, there exist a word $T_n$ of length $t_n$ and positive integers $k, \ell$
such that
$$
Z_n = (T_n)^k \quad \hbox{and} \quad Z'_n = (T_n)^{\ell}.
$$
Consequently, there exists an integer $q$ such that $j_2 = j_1 + q t_n$ and we have 
$$
x_{j_1}^{j_3+v_n-1} = x_{j_1}^{j_2-1} \, x_{j_2}^{j_3+v_n-1} = (T_n)^{q +(1+t)(k + \ell)}.
$$
As above, the letters $x_{j_1+v_n}$ and $x_{j_2+v_n}$
are the same, since $j_1+v_n$ and $j_2+v_n$ are congruent modulo the 
length $t_n$ of $T_n$. This is a contradiction.

We have shown that (3.3) does not hold and we are in position
to complete the proof of the theorem.

Let $(n_k)_{k \ge 1}$ denote the increasing sequence comprising all the integers $n$ 
such that $r(n+1, {\bf x}) \ge r(n, {\bf x}) + 2$. For $k \ge 1$, define $\alpha_k$ and 
$\beta_k$ by setting 
$$
r(n_k, {\bf x}) =  \alpha_k  n_k \quad
\hbox{and} \quad
r(n_k + 1, {\bf x}) = \beta_k  (n_k + 1).
$$
Let $\eps$ be a positive real number with $\eps < \frac{1}{10}$.  
Let $k_0$ be an integer such that $r(n_\ell, {\bf x}) \ge (\rho - \eps) n_{\ell}$ for $\ell \ge k_0$. 
For every integer $k$ greater than $k_0$
and large enough in terms of $\eps$, we have established that $\beta_k \ge \rho +1$ or 
$$
\frac{1 - \beta_k + \alpha_k - \eps}{\beta_k - 1} \le \frac{1+\rho}{(\rho - \eps)^2},  
$$
thus
$$
\beta_k \ge \min\Bigl\{ \rho +1,  
\frac{(\rho - \eps)^2  (\rho + 1 - 2 \eps) + \rho + 1}{1 + \rho + (\rho - \eps)^2} \Bigr\}, 
$$
by using that $\alpha_k \ge \rho - \eps$. 
Since $\eps$ can be taken arbitrarily small, this gives 
$$
\limsup_{n \to + \infty} \, \frac{r(n,{\mathbf x})}{n} 
\ge \min\Bigl\{ \rho +1,  \frac{(\rho + 1) (\rho^2 + 1)}{1 + \rho + \rho^2} \Bigr\}, 
$$
and we have established (3.1). 

Observe that, by definition of the sequence $(n_k)_{k \ge 1}$, 
$$
r(n_{k+1}, {\bf x}) = r(n_k + 1, {\bf x}) + n_{k+1} - n_k - 1 \ge (\rho - \eps) n_{k+1}.
$$
Consequently,
$$
n_{k+1} \le \frac{r(n_k + 1, {\bf x}) - n_k - 1}{\rho - 1 - \eps}.
$$
Let $n$ be an integer with $n_k + 1 \le n \le n_{k+1}$. By (1.2)  
and Lemma \ref{ubound} we have
$$
p(n, {\bf x}) \ge p(n_k + 1, {\bf x}) + n - n_k - 1 \ge r(n_k + 1, {\bf x}) + n - 2 n_k - 2,
$$
thus
$$
\frac{p(n, {\bf x})}{n} \ge 1 + \frac{r(n_k + 1, {\bf x})   - 2 n_k - 2}{n}
\ge 1 + \frac{r(n_k + 1, {\bf x})   - 2 n_k - 2}{n_{k+1}},
$$
giving that
\begin{align*}
\frac{p(n, {\bf x})}{n} & \ge 1 + (\rho - 1 - \eps) \,  
\frac{r(n_k + 1, {\bf x})   - 2 n_k - 2}{r(n_k + 1, {\bf x})   - n_k - 1} \\
& \ge \rho - \eps - (\rho - 1 - \eps) \,  \frac{1}{\beta_k - 1}.
%& \ge \rho - \eps - \frac{(\rho - 1 - \eps) (1 + \rho + (\rho - \eps)^2}{(2 - \alpha_k) (\rho - \eps)^2}. 
\end{align*}
Since $\eps$ can be taken arbitrarily small, we conclude that 
$$
\liminf_{n \to + \infty} \, \frac{p(n, {\bf x})}{n} \ge
\min\Bigl\{ \rho - 1 +  \frac{1}{\rho} , \,  \rho - 1 + \frac{1}{\rho^3} \Bigr\}.   %%d
$$
This proves (3.2) and completes the proof of the theorem.     
\end{proof}

Let $b \ge 2$ be an integer. 
Our last auxiliary result establishes a close connection between the exponent of repetition 
of an infinite word ${\mathbf x}$ written over $\{0, 1, \ldots , b-1\}$ and the
irrationality exponent (see Definition \ref{defirrexp}) 
of the real number whose $b$-ary expansion is given by ${\bf x}$. 

\begin{lemma}\label{minmurep}
Let $b \ge 2$ be an integer and ${\bf x} = x_1 x_2 \ldots$ an infinite word over
$\{0, 1, \ldots , b-1\}$, which is not eventually periodic. 
Then, the irrationality exponent of the irrational number 
$\sum_{k \ge 1} \, \frac{x_k}{b^k}$ satisfies 
$$
\mu \Bigl( \sum_{k \ge 1} \, \frac{x_k}{b^k} \Bigr) \ge \frac{\rep({\mathbf x})}{\rep({\mathbf x}) - 1},
$$
where the right hand side is infinite if $\rep({\mathbf x}) = 1$. 
\end{lemma}

\begin{proof}
Since the irrationality exponent of an irrational real number is at least 
equal to $2$, we can assume that $\rep({\mathbf x}) < 2$. 
Let $n$ and $C$ be positive integers such that $1 < C < 2$ and $r(n,{\mathbf x}) \le C n$. 
By Theorem 1.5.2 of \cite{AlSh03}, 
this implies that there are finite words $W, U, V$ and a positive integer $t$ 
(we do not indicate the dependence on $n$) such that
$|(UV)^t U| = n$ (here and below, $| \cdot |$ denotes the length of a finite word) and
$W (UV)^{t+1} U$ is a prefix of ${\bf x}$ of length at most $C n$.
Observe that 
$$
|W (UV)^{t+1} U| \le Cn \le C | (UV)^t U|,
$$
thus $|WUV| \le (C-1) | (UV)^t U|$. 
Setting $\xi = \sum_{k \ge 1} \, \frac{x_k}{b^k}$, there exists an integer $p$ such that 
$$
\Bigl| \xi - \frac{p}{b^{|W|} (b^{|UV|} - 1)} \Bigr| \le 
\frac{1}{b^{|W (UV)^{t+1} U|}} \le \frac{1}{b^{|WUV|} b^{|WUV|/(C-1)}}.
%\le  \frac{1}{b^{2 |WUV|} b^C }.
$$
Consequently, if there are arbitrarily large integers $n$ with $r(n,{\mathbf x}) \le Cn$, then
$\mu(\xi) \ge 1 + \frac{1}{C-1}$. Since $C$ can be taken arbitrarily close to $\rep({\mathbf x})$,
this implies the theorem. 
\end{proof}

Lemma \ref{minmurep} shows that, when the exponent of repetition
of an infinite word is less than $2$, then the irrationality exponent of the associated real
number exceeds~$2$. We are in position to complete the proof of 
Theorems~\ref{complmain} and \ref{complmainrep}. 

\medskip  

\noindent {\it Proof of Theorems~\ref{complmain} and \ref{complmainrep}.}

Let $b \ge 2$ be an integer and $\xi$ an irrational real number. 
Write $\xi$ in base $b$ as in (1.1) and put 
${\mathbf a} = a_1 a_2 \ldots$. 
Lemma \ref{minmurep} asserts that
$$
\rep({\mathbf a}) \ge \frac{\mu(\xi)}{\mu(\xi) - 1}.
$$
Combined with Theorem \ref{diffReprep}, this gives 
$$
\Rep({\mathbf a})
\ge 1 + \frac{(\rep({\mathbf a}))^3}{1 + \rep({\mathbf a}) + (\rep({\mathbf a}))^2} 
\ge  1 + \frac{\mu^3}{ 3 \mu^3 - 6 \mu^2 + 4 \mu - 1},
$$ 
where $\mu$ denotes the irrationality exponent of $\xi$. 
As well, we obtain
$$
\liminf_{n \to + \infty} \, \frac{p(n, {\mathbf a})}{n} \ge 
\rep({\mathbf a})  - 1 + \frac{1}{(\rep({\mathbf a}))^3}     
\ge \frac{\mu^4 - 3 \mu^3 + 6 \mu^2 - 4 \mu +1 }{\mu^3 (\mu - 1)}.
$$
We have established (1.3) and (2.1) and thereby completed the proofs of 
Theorems~\ref{complmain} and \ref{complmainrep}. 

\bigskip
\goodbreak

\noindent {\it Additional comments.}

We can slightly improve Theorem~\ref{diffReprep} (and, consequently, 
Theorems~\ref{complmain} and \ref{complmainrep}) by means of a 
refined case-by-case analysis. 
With the notation used in the proof of Theorem~\ref{diffReprep}, 
the two cases to distinguish are:

(i) $j_1 = m_2$ and $j_2 = m_1 + n - v_n$ (that is, $m_2 < m_1 + n - v_n$);

(ii) $j_1 = m_1 + n - v_n$ and $j_2 = m_2$ (that is, $m_2 > m_1 + n - v_n$). 

\noindent
Then, (3.1) can be replaced by the stronger inequality which holds for Case (i)  
$$
\Rep({\mathbf x}) \ge \rep({\mathbf x}) + \frac{1}{\rep({\mathbf x}) + (\rep({\mathbf x}))^2} 
\eqno (3.9)
$$
and (2.1) by
$$
\limsup_{n \to + \infty} \, \frac{r(n, \xi, b)}{n}  
\ge  2 + \frac{2 \mu^2 + \mu - 1 - \mu^3}{\mu (\mu - 1) (2 \mu - 1)}.      
$$ 
Furthermore, we may also see that, under a slightly weaker assumption than (3.9), 
Case (i) cannot occur 
for two consecutive integers $n$ such that $r(n+1, {\bf x}) \ge r(n, {\bf x}) + 2$. 
Hence, a further small improvement can be obtained.  

\medskip
\section*{Acknowledgement}
Dong Han Kim was supported by the National Research Foundation of Korea (NRF-2015R1A2A2A01007090). 

\goodbreak

\end{document}